\documentclass[12pt]{amsart}

\usepackage{amssymb}

\newtheorem{Thm}{Theorem}[section]

\newtheorem{Cor}[Thm]{Corollary}
\newtheorem{Lem}[Thm]{Lemma}
\newtheorem{Prop}[Thm]{Proposition}

\theoremstyle{definition}

\theoremstyle{remark}

   \renewcommand{\hat}[1]{\overline{#1}}

\def \eps{\varepsilon}

\def\Ndb{\mathbb N}

\def\Rdb{\mathbb R}

\begin{document}
\title{Isometric embeddings of compact spaces into Banach spaces}

\author {Y. Dutrieux}
\address{Universit\'e de Franche-Comt\'e, Laboratoire de Math\'ematiques UMR 6623,
16 route de Gray, 25030 Besan\c con Cedex, FRANCE.}
\email{yves.dutrieux@univ-fcomte.fr}

\author {G. Lancien}
\address{Universit\'e de Franche-Comt\'e, Laboratoire de Math\'ematiques UMR 6623,
16 route de Gray, 25030 Besan\c con Cedex, FRANCE.}
\email{gilles.lancien@univ-fcomte.fr}

\begin{abstract}
We show the existence of a compact metric space $K$ such that whenever $K$
embeds isometrically into a Banach space $Y$, then any separable Banach space
is linearly isometric to a subspace of $Y$. We also address the following
related question: if a Banach space $Y$ contains an isometric copy of the unit
ball or of some special compact subset of a separable Banach space $X$, does it
necessarily contain a subspace isometric to $X$? We answer positively this
question when $X$ is a polyhedral finite-dimensional space, $c_0$ or $\ell_1$.
\end{abstract}


\maketitle

\section{Introduction}\label{intro}
This paper is motivated by questions about universal Banach spaces. In 1927,
P.S.~Urysohn \cite{U} was the first to give an example of a separable metric
space $\mathbb U$ such that every separable metric space is isometric to a
subset of $\mathbb U$ (we say that $\mathbb U$ is isometrically universal).
However the foundation of the questions about universal Banach spaces is the
theorem of S.~Banach and S.~Mazur \cite{B} asserting that every separable
Banach space is linearly isometric to a subspace of $C([0,1])$ and therefore
every separable metric space is isometric to a subset of $C([0,1])$. It is then
natural to wonder what are the Banach spaces that are isometrically universal
for smaller classes of Banach spaces or metric spaces. For instance,
G.~Godefroy and N.J.~Kalton proved very recently in \cite{GK2} that if a
separable Banach space contains an isometric copy of every separable strictly
convex Banach space, then it contains an isometric copy of every separable
Banach space. On the other hand, in another recent work, N.J.~Kalton and the
second named author showed that every metric space with relatively compact
balls embeds almost isometrically into the Banach space $c_0$. The main result
of this paper is that a Banach space containing isometrically every compact
metric space must contain a subspace linearly isometric to $C([0,1])$.

The techniques that we use come from classical results on isometries between
Banach spaces. The first of them is of course the well known result of S.~Mazur
and S.~Ulam \cite{MU} who proved that a surjective isometry between two Banach
spaces is necessarily affine. In other words, the linear structure of a Banach
space is completely determined by its isometric structure. Then, one naturally
wonders about what can be said when a Banach space $X$ is isometric to a subset
of a Banach space $Y$. The first fundamental result in this direction is due to
T.~Figiel who showed in \cite{F} that if $j:X\to Y$ is an isometric embedding
such that $j(0)=0$ and $Y$ is the closed linear span of $j(X)$, then there is
linear quotient $Q:Y\to X$ of norm one and so that $Q\circ j=Id_X$. More recently, as
an application of their work on Lipschitz-free Banach spaces, G.~Godefroy and
N.J.~Kalton \cite{GK} could use Figiel's result to prove that if a separable
Banach space is isometric to a subset of another Banach space $Y$, then it is
actually linearly isometric to a subspace of $Y$. Let us mention that this is
not true in the non separable case and that counterexamples are given in
\cite{GK}.

In section 2 we recall the necessary background on Lipschitz-free Banach
spaces. We also state the version of Theorem 3.1 of \cite{GK} that we shall use
in the sequel. In section 3 we prove the main result of the paper. More
precisely, we produce a compact subset $K_0$ of $C([0,1])$ such that any Banach
space containing an isometric copy of $K_0$ must contain a subspace which is
linearly isometric to $C([0,1])$. We also show how our technique can be
combined with the results of G.M.~L\"ovblom in \cite {L} on almost isometries
between $C(K)$-spaces.

Finally, let us say that $M$ is an {\it isometrically representing subset} of
the Banach space $X$ if any Banach space $Y$ containing an isometric copy of
$M$ contains a subset which is isometric to $X$. Notice that if $M$ is an
isometrically representing subset of a separable Banach space $X$, then it
follows from the result of Godefroy and Kalton that any Banach space containing
an isometric copy of $M$ has a subspace which is linearly isometric to $X$. In
the last section we produce compact isometrically representing subsets for the
finite dimensional polyhedral spaces and for $\ell_1$. We also show that the
unit ball of $c_0$ isometrically represents the whole space.

\section{Preliminary results}\label{Preliminary}

We begin this section with a localized version of Theorem 3.1 and Corollary 3.3
in \cite{GK}. We use the same notation as in \cite{GK} but recall them for the
seek of completeness.

Let $(E,d)$ be a metric space with a specified point that we denote as $0$. For
$Y$ a Banach space and $f: E\to Y$ we write
$$\|f\|_L = \sup\left\{\frac{\|f(y)-f(x)\|}{d(x,y)}\; ;\; x\neq y \text{ in } E\right\}$$
The space $Lip_0(E)$ is the space of all $f: E\to \mathbb R$ such that $f(0)=0$
and $\|f\|_L<\infty$ equipped with the norm $\|\cdot\|_L$. It turns to have a
canonical predual $\mathcal F(X)$ which is the closed linear span in the dual
of $Lip_0(E)$ of the evaluation functionals $\delta(x)$ defined by
$\delta(x)(f)=f(x)$, for all $f$ in $Lip_0(E)$ and $x$ in $E$. If $Y$ is a
Banach space and $g: E\to Y$ is a Lipschitz map, then there exists a unique
linear operator $\hat g: \mathcal F(E)\to Y$ such that $\hat g\circ \delta =
g$. Moreover $\| \hat g\| = \|g\|_L$. In particular, when $E$ is a Banach
space, applying this to the identity map on $E$, we see that $\delta$ admits a
norm-one linear left inverse $\beta$.

When $F$ is a subset of $E$ which contains $0$,  we denote by $\mathcal F_E(F)$
the closed linear span in $\mathcal F(E)$ of the evaluation functionals
$\delta(x)$, $x\in F$. Since, by inf-convolution, any real valued Lipschitz
function on $F$ can be extended to the whole space $E$ with the same Lipschitz
constant, it is clear that the spaces $\mathcal F_E(F)$ and $\mathcal F(F)$ are
canonically isometric.

\medskip In our first lemma, we rephrase Theorem 3.1 of \cite{GK} for our
particular purpose.

\begin{Lem}\label{GodKal}
Let $X$ be a separable Banach space. Let $F$ be a closed convex subset of $X$
such that $0\in F$. We assume that the closed linear span of $F$ is $X$. Then
there exists an isometric linear embedding $T: X\to \mathcal F_X(F)$ such that
$\beta\circ T$ is the identity map on $X$.
\end{Lem}

\begin{proof}
Since $X$ is separable, there exists a sequence $(x_n)_{n\ge 1}$ in $F$ which
is total in $X$ and such that the set $\{\sum_{k=1}^\infty t_kx_k\;;\; 0\le
t_k\le 1 \text{ for all }k\}$ is a compact subset of $F$. We introduce the
Hilbert cube 
$$H_n=\big\{(t_k)_{k=1}^\infty\;;\; 0\le t_k\le 1 \text{ for all }k
\text{ and }t_n=0\big\}$$
endowed with the product Lebesgue measure $\lambda_n$.
Following the proof of Theorem 3.1 in \cite{GK}, we define
$$\phi_n = \int_{H_n} \left[\delta\left(x_n+\sum t_k x_k\right) -
\delta\left(\sum t_kx_k\right)\right]d\lambda_n(t)$$ Our choice of $(x_n)$
ensures that $\phi_n\in \mathcal F_X(F)$. As proved in Theorem 3.1 in
\cite{GK}, the map $x_n\mapsto \phi_n$ extends to a norm-one linear operator
$T$ from $X$ to $\mathcal F_X(F)$ such that $\beta\circ T$ is the identity map
on $X$.
\end{proof}
{From this, we derive the main statement of this section.}

\begin{Thm}\label{godkalfig}
Let $X$ and $Z$ be Banach spaces. Assume that $X$ is separable. Let $F$ be a
closed convex subset of $X$, containing $0$ and such that the closed linear
span of $F$ is $X$. Let $j: F\to Z$ be an isometric embedding such that
$j(0)=0$. Assume also that there exists a linear operator $Q: Z\to X$
satisfying $\|Q\|\le 1$ and $(Q\circ j)(x)=x$ for any $x\in F$. Then $X$ is
linearly isometric to a subspace of $Z$.
\end{Thm}

\begin{proof}
Let $T : X\to \mathcal F_X(F)$ be the operator given by Lemma \ref{GodKal}. Let
$\hat \jmath : \mathcal F_X(F)\to Z$ be the linear operator defined by $\hat
\jmath\circ \delta = j$. For any $x\in F$, we have $Q\circ \hat \jmath\circ\delta(x) =
Q\circ j(x)=x=\beta\circ\delta(x)$. Hence, by linearity and continuity, $Q \circ\hat
\jmath(\mu)=\beta(\mu)$ for $\mu\in \mathcal F_X(F)$. Since
$T(X)\subset\mathcal F_X(F)$, we have $Q\circ \hat \jmath\circ T(x)=\beta\circ T(x)=x$ for
any $x\in F$ and thus, again by linearity and continuity, for any $x\in X$.
Finally, the fact that $Q$ is a contraction implies that $\hat \jmath\circ T :
X\to Z$ is a linear isometric embedding.
\end{proof}

\section{Isometric Embeddings of spaces of continuous functions}\label{C(K)spaces}

We begin this section with the main result of this paper.

\begin{Thm}\label{lip} Let $(R,d)$ be a compact metric space. Then  there is
a compact subset $K$ of $C(R)$ such that whenever $K$ embeds isometrically into
a Banach space $Y$, then $C(R)$ is linearly isometric to a subspace of $Y$.
\end{Thm}

\begin{proof} We may assume that the diameter of $(R,d)$ is less than or equal to 1. Then
we consider $K=\{f\in C(R), \ \|f\|_\infty\le 1 \ {\rm and}\ \|f\|_L\le 1\}$.
Let $Y$ be a Banach space and assume that $j: K\to Y$ is an isometry such that
$j(0)=0$. We denote $F=K/5$ and $Z$ the closed linear span in $Y$ of $j(F)$. In
view of Theorem \ref{godkalfig}, it is enough to build a continuous linear map
$Q: Z\to C(R)$ such that $\|Q\|\le 1$ and for any $\varphi \in F$, $(Q\circ
j)(\varphi)=\varphi$. Our construction is adapted from a work of T. Figiel
\cite{F} that was already used in \cite{GK}.

\noindent For $s,t \in R$, we define $\varphi_t(s)=1-d(s,t)$. For $t$ in $R$,
the functions $\varphi_t$ and $-\varphi_t$ clearly belong to $K$ and
$\|j(\varphi_t)-j(-\varphi_t)\| =2$. Thus we can pick $y^*_t\in Y^*$ such that
$\|y^*_t\|=1$ and $y^*_t(j(\varphi_t)-j(-\varphi_t))=2$. Since $j$ is an
isometry and $j(0)=0$, we clearly have:
\begin{equation}\label{normingfunctional}
\forall t\in R\ \ \forall \lambda\in[-1,1]\ \ (y^*_t \circ j)(\lambda
\varphi_t)=\lambda.
\end{equation}

To conclude our proof, it will be enough to show that
\begin{equation}\label{keypoint}
\forall t\in R\ \ \forall \varphi \in F\ \ (y^*_t \circ j)(\varphi)=\varphi(t).
\end{equation}
Indeed, we could then define for $y \in Z$, $Q(y)=(y^*_t(y))_{t\in R}$. It
clearly follows from (\ref{keypoint}) that $Q$ is a continuous linear map from
$Z$ to $C(R)$ such that $\|Q\|\le 1$ and $(Q\circ j)(\varphi)=\varphi$ for all
$\varphi \in F$.

\noindent So let us assume that there exist $t\in R$ and $\varphi \in F$ such
that $(y^*_t \circ j)(\varphi)\neq \varphi(t)$. We set $\psi=\varphi-(y^*_t
\circ j)(\varphi)\varphi_t$. Since $\|\psi\|_\infty < 1/2$ and $\psi(t)\neq
0$, there exists $u\in \{-2,2\}$ so that
$$0< (\varphi_t-u\psi)(t)<1.$$
Besides, $\|\psi\|_L< 1/2$ and the diameter of $R$ is less than or equal to 1, so for any
$s\in R$:
$$-1<1-d(s,t)-u\psi(s)=(\varphi_t-u\psi)(s)\le(\varphi_t-u\psi)(t)<1.$$
Hence
\begin{equation}\label{ineq1}
\|\psi-\frac{\varphi_t}{u}\|_\infty=\|\varphi-\big( (y^*_t\circ
j)(\varphi)+\frac{1}{u}\big)\varphi_t\|_\infty<\frac{1}{2}\cdot
\end{equation}
Note that $\lambda=(y^*_t\circ j)(\varphi)+\frac{1}{u} \in [-1,1]$. So
(\ref{normingfunctional}) yields:
$$\frac{1}{2}=|(y^*_t\circ
j)(\varphi)-\lambda|=|(y^*_t\circ j)(\varphi)-(y^*_t\circ
j)(\lambda\varphi_t)|\le \|\varphi-\lambda \varphi_t\|_\infty.$$ This is in
contradiction with inequality  (\ref{ineq1}).

\end{proof}
{From the universality of $C([0,1])$, we immediately deduce the following.}

\begin{Cor} Consider the following compact subset of $C([0,1])$:
$$K_0=\{f\in C([0,1]), \ \|f\|_\infty\le 1 \ {\rm and}\ \|f\|_L\le 1\}.$$ If a Banach
space $Y$ contains an isometric copy of $K_0$, then it contains an isometric
copy of any separable metric space and any separable Banach space is linearly
isometric to a subspace of $Y$.
\end{Cor}

It is now natural to ask if a metric space that is isometrically universal for
all metric compact spaces is isometrically universal for all separable metric
spaces. The next proposition shows that, for elementary reasons, this is not
the case.

\begin{Prop} There exists a separable metric space $V$ such that every
separable and bounded metric space is isometric to a subset of $V$ but so that
$\Rdb$ cannot be isometrically embedded into $V$.
\end{Prop}

\begin{proof} Let $B$ denote the unit ball of $C([0,1])$. Notice first that
$rB$ contains an isometric copy of all separable metric spaces with diameter
less than $r$. Let $V$ be the disjoint union of the sets $V_n=nB$, for $n\in
\Ndb$. We now define a metric $d$ on $V$ as follows. On $V_n$, $d$ is the
natural distance in $C([0,1])$. For $n\neq m$, $f\in V_n$ and $g\in V_m$, we
set $d(f,g)= \|f\|_\infty +1+\|g\|_\infty$. It is clear that $(V,d)$ is a
separable metric space which is universal for all separable bounded metric
spaces. On the other hand, any connected component of $V$ is bounded. Therefore
$\Rdb$ does not embed isometrically into $V$.
\end{proof}

We shall now combine Theorem \ref{lip} with a result of G.M. L\"{o}vblom \cite{L}
on almost isometries between $C(K)$ spaces to obtain

\begin{Cor}
Let $R$ and $S$ be compact metric spaces. Assume there exists a Lipschitz
embedding $F$ of the unit ball of $C(R)$ into $C(S)$ such that
$$\forall f,g \in B_{C(R)}\ \ \frac{15}{16}\,\|f-g\|\le \|F(f)-F(g)\| \le
\|f-g\|.$$ Then $C(R)$ is linearly isometric to a subspace of $C(S)$.
\end{Cor}

\begin{proof}
We may assume that $F(0)=0$. Using Theorem 2.1 in \cite{L} for the particular
value of $\eps=\frac{1}{16}$ , we obtain an isometry $j :
\frac{1}{2}B_{C(R)}\to B_{C(S)}$. In particular, the set of functions on $R$
such that both the supremum and the Lipschitz norms are less than or equal to
$\frac{1}{2}$ isometrically embeds into $C(S)$. Then it follows from our
previous proof that $C(R)$ embeds linearly isometrically into $C(S)$.
\end{proof}

\section{Isometrically representing subsets}

In this section, we address the following problem: given a separable Banach
space $X$, we look for a small subset $K$ of $X$ such that whenever $K$
isometrically embeds into a Banach space $Y$, then $X$ embeds linearly
isometrically into $Y$. We remind the reader that we call such a set $K$ an
{\it isometrically representing subset} of $X$. We shall restrict ourselves to
considering $K$ to be a compact subset of $X$ or the unit ball of $X$.

\medskip We start with a finite dimensional result.

\begin{Thm}\label{polytope} Let $X$ be a finite dimensional polyhedral
Banach space. Then the unit ball of $X$ is an isometrically representing subset
of $X$.
\end{Thm}

\begin{proof} Let $j:B_X \to Y$ be an isometric embedding such that $j(0)=0$.

\noindent Let $x^*_1,..,x^*_k \in S_{X^*}$ so that $B_X =\cap_{i=1}^k \{x\in
X,\ |x^*_i(x)|\le 1\}$. After removing some of the $x^*_i$'s if necessary, we
may and do assume that
$$\forall i \in \{1,..,k\}\ \ \exists x_i\in S_X,\ \ x^*_i(x_i)=1\ {\rm and}\
\forall l\neq i\ |x^*_l(x_i)|<1.$$ Then
\begin{equation}\label{goodball}
\exists r\in (0,\frac{1}{2}]\ \forall x\in X\ \forall i\in\{1,..,k\} \ \
\|x-x_i\|\le 4r \Rightarrow \|x\|=x_i^*(x).
\end{equation}
We now imitate the proof of Theorem \ref{lip} with the $x_i$'s playing the role
of the functions $\varphi_t$. So for $1\le i\le k$, we can pick $y_i^*\in
S_{Y^*}$ so that for any $\lambda\in[-1,1]$, $(y^*_i\circ j)(\lambda
x_i)=\lambda$. Assume now that $x\in rB_X$ and $(y^*_i\circ j)(x)\neq
x^*_i(x)$. Then consider $w=x-(y^*_i\circ j)(x)x_i$. Since $\|2w\|\le 4r$ and
$x^*_i(w)\neq 0$, it follows from (\ref{goodball}) that there is $u\in\{-2,2\}$
such that $x_i^*(x_i-uw)=\|x_i-uw\|<1$. Following the lines of our previous
proof, we then get a contradiction. Thus we have
\begin{equation}\label{keypoint2}
\forall x\in rB_X\ \ \forall i\in\{1,..,k\}\ \ (y^*_i\circ j)(x)=x^*_i(x).
\end{equation}
It is clear that $\{x^*_i\}_{i=1}^k$ spans $X^*$. So, if $n$ is the dimension
of $X$, we can pick $i_1<..<i_n$ such that $(x_{i_1}^*,..,x_{i_n}^*)$ is a
basis of $X^*$. Denote $\{z_1,..,z_n\}$ the basis of $X$ whose dual basis is
$(x_{i_1}^*,..,x_{i_n}^*)$. Now define
$$\forall y\in Y\ \ Q(y)=\sum_{m=1}^n y^*_{i_m}(y)z_m.$$
This map is linear and continuous from $Y$ to $X$. It follows from
(\ref{keypoint2}) that $(Q \circ j)(x)=x$ for all $x\in rB_X$. Let now $Z$ be
the closed linear span of $j(rB_X)$. We also get from (\ref{keypoint2}) that
$$\forall y \in Z\ \ \forall i\in\{1,..,k\}\ \ (x^*_i\circ Q)(y)=y^*_i(y).$$
So, if $y\in Z$, $\|Q(y)\|=\sup_i|x^*_i(Qy)|=\sup_i|y^*_i(y)|\le \|y\|$. Hence,
we are in situation to apply Theorem \ref{godkalfig} and conclude that $X$ is
linearly isometric to a subspace of $Z$.

\end{proof}

As a consequence, we obtain a similar result for $c_0$.

\begin{Cor} The Banach space $c_0$ is isometrically represented by its unit ball.
\end{Cor}

\begin{proof} Let $(e_k)_{k\ge 1}$ be the canonical basis of $c_0$, $(e^*_k)_{k\ge 1}$
the associated linear functionals and $X_n$ the linear span of
$\{e_1,..,e_n\}$. Assume that $j:B_{c_0} \to Y$ is an isometric embedding. Then
let $Z_n$ be the closed linear span of $j(\frac{1}{8} B_{X_n})$ and $Z$ be the
closed linear span of $j(\frac{1}{8} B_{c_0})$. Notice that for any $x\in c_0$
satisfying $\|x-e_k\|\le \frac{1}{2}$, we have that $\|x\|=e^*_k(x)$. Then it
follows from our previous argument that there exists a continuous linear map
$Q_n:Z_n\to c_0$ such that $\|Q_n\|\le 1$ and $(Q_n\circ j)(x)=x$ for all $x\in
\frac{1}{8}B_{X_n}$. This clearly yields the existence of $Q:Z\to c_0$ so that
$\|Q\|\le 1$ and $(Q\circ j)(x)=x$ for any $x$ in $\frac{1}{8}B_{c_0}$. The
conclusion follows again from Theorem \ref{godkalfig}.
\end{proof}

In the case of $\ell_1$, we need to use a completely different method to obtain
the following.

\begin{Prop} The Banach space $\ell_1$ admits a compact isometrically representing subset.
\end{Prop}

\begin{proof}
For $A\subset \mathbb N$, we define $\mu(A) = \sum_{k\in A} 2^{-k}=\|\sum_{k\in
A} 2^{-k}e_k\|_{\ell_1}$ where $(e_k)$ stands for the canonical basis of
$\ell_1$. We denote by $K$ the space of all subsets of $\mathbb N$ endowed with
the metric  $d(A,B) = \mu(A\backslash B)+\mu (B\backslash A)$. It is clear that
$K$ is isometric to a compact subset of $\ell_1$.

Assume now that $j: K\to Y$ is an isometric embedding of $K$ into some Banach
space $Y$. We may assume that $j(\emptyset)=0$. For any $n\in\mathbb N$, we set
$y_n = 2^n j(\{n\})$. Notice that $\|y_n\|=1$. For any $\alpha = (\alpha_k)\in
\ell_1$, we define $T\alpha = \sum \alpha_k y_k\in Y$. $T$ is clearly a
norm-one operator. Moreover, given $\alpha\in \ell_1$, we put $P=\{k\;;\;
\alpha_k>0\}$, $Q=\{k\;;\; \alpha_k<0\}$ and $y^*\in S_{Y^*}$ such that
$y^*(j(P)-j(Q))=d(P,Q)$. Since all the triangle inequalities are equalities, we
infer that $y^*(j(\{p\}))=2^{-p}$ and $y^*(j(\{q\}))=-2^{-q}$ for any $p\in P$,
$q\in Q$. Hence $y^*(T\alpha) = \sum |\alpha_k| $ and $T$ is a linear isometric
embedding.
\end{proof}

\noindent {\bf Questions.} We leave open the following questions:

(1) Is a Banach space always isometrically represented by its unit ball?

(2) Does every separable Banach space admit a compact isometrically
representing subset?

\end{document}